\documentclass{amsart}
\usepackage[margin=1in]{geometry}
\usepackage{color}

\newtheorem{theorem}{Theorem}

\newtheorem*{remark}{Remark}
\newtheorem{lemma}{Lemma}
\newcommand{\R}{\mathbb R}
\newcommand{\eps}{\varepsilon}
\newcommand{\sgn}{\mbox{sgn}}
\newcommand{\dd}{\, \mathrm{d}}

\title{An integro-differential equation without continuous solutions}
\author{Luis Silvestre}
\author{Stanley Snelson}
\thanks{LS was partially supported by NSF grant DMS-1254332. SS was partially supported by NSF grant DMS-1246999.}
\begin{document}

\begin{abstract}
We show an example of a non symmetric integro-differential equation of order $\alpha$, for $\alpha \in (0,1)$, for which H\"older estimates do not hold even though the kernels are comparable to the fractional Laplacian.
\end{abstract}

\maketitle

\section{Introduction}

We are concerned with an integro-differential equation of the usual form
\begin{equation} \label{e:main}  \int_{\R^n} \left( u(x+y) - u(x) - y \cdot \nabla u(x) \chi_{B_1}(y) \right) K(x,y) \dd y = f(x) \ \text{ in } B_1.
\end{equation}
The purpose of this article is to show an example of a kernel $K(x,y)$ satisfying, for some $\alpha \in (0,1)$,
\begin{equation} \label{e:kernel-bounds}  \frac \lambda {|y|^{n+\alpha}} \leq K(x,y) \leq \frac \Lambda {|y|^{n+\alpha}}.
\end{equation}
and a bounded function $f$, for which the solution $u$ of \eqref{e:main} does not satisfy any modulus of continuity a priori in terms of $\|u\|_{L^\infty}$ and $\|f\|_{L^\infty}$.

The key of our example is that we do not make the symmetry assumption $K(x,y)=K(x,-y)$. The correction term $- y \cdot \nabla u(x) \chi_{B_1}(y)$ inside the integrand effectively creates a drift term. Since we take $\alpha<1$, the regularization effect of the symmetric part of the integral does not compensate the effect of this implicit drift. Any modulus of continuity can thus be invalidated following a mechanism similar to that in \cite{silvestre2013loss}.

When the kernel $K$ satisfies the symmetry assumption $K(x,y) = K(x,-y)$, then the solutions to \eqref{e:main} satisfy a regularity estimate in H\"older spaces
\[ \|u\|_{C^\gamma(B_{1/2})} \leq C \left( \|u\|_{L^\infty(\R^n)} + \|f\|_{L^\infty(B_1)} \right),\]
for some $\gamma > 0$. This estimate was obtained first by R. Bass and D. Levin \cite{bass2002harnack}. H\"older estimates of this type have been a topic of high interest in recent years, with several results in this direction for different types of integral equations including \cite{song2004harnack}, \cite{bass2005holder}, \cite{silvestre2006holder}, \cite{caffarelli2009regularity}, \cite{chang2012regularity}, \cite{lara2012regularity}, and also results for parabolic integral equations like \cite{lara2014regularity}, \cite{silvestre2011differentiability} and \cite{schwab2014regularity}.

In some cases, H\"older estimates hold for non symmetric kernels $K(x,y)$. That is the case of the results in \cite{chang2012regularity}, \cite{lara2012regularity} and \cite{schwab2014regularity}. In those cases, for $\alpha < 1$, the equation has to be taken without the gradient correction term. That is, the equation is
\[
\int_{\R^n} \left( u(x+y) - u(x) \right) K(x,y) \dd y = f(x) \ \text{ in } B_1.
\]
Note that for $\alpha \in (0,1)$, the left hand side makes sense for every function $u \in C^1$. While equation \eqref{e:main} is of the form that appears traditionally in the probability literature, it is better not to include the gradient correction term in this case. The result in this note shows that in fact, this correction term ruins any continuity estimate. There have been attempts to obtain H\"older continuity estimates for equations of this form, see for example \cite{kim2012regularity}.

Now we state our main result.

\begin{theorem} \label{t:main}
For any $\alpha \in (0,1)$ and $0<\lambda<\Lambda$, and any modulus of continuity $\eta$, there is a kernel $K(x,y)$ satisfying \eqref{e:kernel-bounds}, and a function $u : \R^n \to \R$ such that
\begin{itemize}
\item $u(x) \in [-1,1]$ for all $x \in \R^n$.
\item $u$ solves the equation \eqref{e:main} with kernel $K(x,y)$ and $f(x)\equiv 0$.
\item The function $u$ does not obey the modulus of continuity $\eta$ at the origin.
\end{itemize}
\end{theorem}

\begin{remark}
The solution $u$ constructed in the next section is continuous on $\R^n$, smooth on the set $\{x_1 \neq 0\}$, and solves \eqref{e:main} classically on $B_1\setminus\{x_1=0\}$. On the line $\{x_1 =0\}$, $\nabla u(x)$ does not exist, but the equation is satisfied in the viscosity sense, as there are no $C^2$ functions touching $u$ from above or below. More precisely, let $M^\pm$ be the extremal operators
\begin{align*}
M^+ u(x) = \sup \left \{  \int_{\R^n} \left( u(x+y) - u(x) - y \cdot \nabla u(x) \chi_{B_1}(y) \right) K(y) \dd y \ : \ \frac \lambda {|y|^{n+\alpha}} \leq K(y) \leq \frac \Lambda {|y|^{n+\alpha}} \right\}, \\
M^- u(x) = \inf \left \{  \int_{\R^n} \left( u(x+y) - u(x) - y \cdot \nabla u(x) \chi_{B_1}(y) \right) K(y) \dd y \ : \ \frac \lambda {|y|^{n+\alpha}} \leq K(y) \leq \frac \Lambda {|y|^{n+\alpha}} \right\}.
\end{align*}
Then, the function $u$ from Theorem \ref{t:main} satisfies $M^+ u \geq 0$ and $M^- u \leq 0$ in the viscosity sense in $B_1$.

See \cite{kim2012regularity} for a more explicit expression of the operators $M^+$ and $M^-$.
\end{remark}

\section{The proof}

We give the value of the kernel $K(x,y)$ below. The function $u$ will depend on the variable $x_1$ only. Thus, our example is essentially one-dimensional, and we will use the notations $u(x)$ and $u(x_1)$ interchangeably. The stategy of the proof is as follows: first, construct a family of bounded, continuous functions $u_r$ that approximate the discontinuous function $u_0(x_1)= \sgn(x_1)$ as $r\rightarrow 0$. If $r>0$ is chosen small enough, $u_r$ will fail to admit a given modulus of continuity. Next, we will add a continuous, increasing function $v(x_1)$ to $u_r(x_1)$, with $v$ independent of $r$, to ensure that the quantity $L(u_r+v)(x)$, with
\begin{equation}\label{e:L}
L u(x) = \int_{\R^n} (u(x+y)-u(x)-y\cdot \nabla u(x)\chi_{B_1}(y))K(x,y)\dd y,
\end{equation}
satisfies $-C_0 \leq L (u_r+v)(x) \leq C_0$ for $x\in B_1$, with $C_0$ independent of $r$. (See Lemma \ref{l:bounded}.) Since $v$ is increasing in $x_1$, the function $u_r+v$  will still break the given modulus of continuity. Finally, we will add another continuous, increasing function $w(x_1)$ to our solution (see Lemma \ref{l:w}), such that the sum $u_r(x_1)+v(x_1)+w(x_1)$ satisfies $L(u_r+v+w)(x_1) = 0$ for $x_1\in [-1.1]$, and $u_r+v+w$ will also break the given modulus of continuity.

For $0<\alpha<1$ and $0<\lambda<\Lambda$, let us define the kernel $K(x,y)$ as follows: let 
\begin{align*}
K_1(x,y) &= \frac{(\lambda+\Lambda)/2}{|y|^{n+\alpha}},\\
K_2(x,y) &= \sgn(y_1)\frac{\Lambda-\lambda}2\chi_{B_1}(y),\\
K_3(x,y) &= \sgn(y_1)\frac{\Lambda-\lambda}2 \frac 1 {|y|^{n+\alpha}}\chi_{\R^n\setminus B_1}(y).
\end{align*}
Note that $K_1(x,y)$ is even in $y$, and $K_2(x,y)$ and $K_3(x,y)$ are odd. Our kernel is defined by
\begin{equation}\label{e:kernel}
K(x,y) :=  K_1(x,y) + a(x) K_2(x,y) - c(x) K_3(x,y),
\end{equation}
where $a(x)$ is to be chosen later. The condition (\ref{e:kernel-bounds}) implies that we need $|a(x)|\leq 1$ and $|c(x)|\leq 1$ for all $x$.
Let $b(x)$ be the drift vector, which is given by
\[b(x) = \int_{B_1} yK_3(x,y)\dd y = \frac{\Lambda-\lambda}2\int_{B_1} y\,\sgn(y_1)\dd y = \frac{\Lambda-\lambda}2\,\frac{\omega_n}2  e_1,\]
where $\omega_n$ is a constant depending on the dimension $n$ only.

We define 
\begin{align}\label{e:L123}
L_i u(x) &=  \int_{\R^n} (u(x+y) - u(x))K_i(x,y)\dd y, \quad i = 1,2,3,\\ 
L_4 u(x) &= b(x) \cdot \nabla u(x).\nonumber
\end{align}
With this notation, 
\begin{equation}\label{e:operator}
Lu(x) = L_1u(x) + a(x)(L_2u(x) - L_4 u(x))  - c(x) L_3 u(x).
\end{equation}
Note that $L_1$ is a multiple of the usual fractional Laplacian: $L_1 u = -c_{n,\alpha} (-\Delta)^{\alpha/2} u$ for some constant $c_{n,\alpha}>0$,
and for $u$ depending only on $x_1$, $L_4 = C_1\partial_{x_1}$, where $C_1 = \omega_n(\Lambda-\lambda)/4$. It is well known that the gradient and the fractional Laplacian have the following scaling: if $u_r(x) = u(x/r)$, then
\begin{equation}\label{e:scaling}
L_1 u_r(x) = r^{-\alpha}[L_1 u](x/r) \mbox{ and } L_4 u_r(x) = r^{-1}[L_4 u](x/r),
\end{equation}
and we will repeatedly make use of this. Furthermore, simple integral estimates show that $L_2 u(x)$ and $L_3 u(x)$ are bounded in $B_1$ for any function $u$ that is bounded in $\R^n$.

The following lemma establishes a solution $u$ to equation (\ref{e:main}) with bounded right-hand side $f$, such that $u$ breaks a given modulus of continuity.

\begin{lemma}\label{l:bounded}
For any $\alpha\in(0,1)$, $0<\lambda<\Lambda$, and modulus of continuity $\eta$ at the origin, there exist a bounded function $u:\R^n\rightarrow\R$ and a kernel of the form \eqref{e:kernel}, with $c(x)=0$, such that 
\begin{equation}\label{e:bounded}
-C_0\leq Lu(x)\leq C_0, \quad x\in B_1,
\end{equation}
with $L$ as in \eqref{e:L}, and $u$ breaks $\eta$. The function $u$ depends only on $x_1$ and is monotonically increasing. The constant $C_0$ and $\sup_{\R^n} |u|$ depend on $\alpha$, $\lambda$, $\Lambda$, and $n$, but are independent of $\eta$.
\end{lemma}

\begin{proof} Let $u_1(x_1)$ be any smooth, nondecreasing function such that $u_1(x_1) = -1$ for $x_1\leq -1$ and $u_1(x_1)= 1$ for $x_1\geq 1$, and define $u_r(x_1) = u_1(x_1/r)$. Choose $r>0$ small enough that $\eta(r)<1$. Then the oscillation $\omega_{u_r}(B_r) = \sup_{B_r} u_r(x) - \inf_{B_r} u_r(x) = 1 \geq \eta(r)$, so $u_r(x)$ breaks $\eta$.

The function $u$ satisfying (\ref{e:bounded}) will be of the form $u(x) = u_r(x_1) + v(x_1)$, with $v$ another function depending only on $x_1$ and increasing in $x_1$. To define $v$, we pick a small $\eps>0$ (such that $\alpha+\eps<1$) and set
\[v(x_1) = \begin{cases} -2^{1-\alpha-\eps}, &x_1<-2,\\
							\sgn(x_1)|x_1|^{1-\alpha-\eps}, &|x_1|\leq 2,\\
							2^{1-\alpha-\eps}, &x_1>2.\end{cases}\]
Since $v$ is nondecreasing in $x_1$, the function $u(x) = u_r(x) + v(x)$ also breaks the modulus of continuity $\eta$.

We claim there exists $\delta>0$, independent of $r$, such that $L_4 u(x_1) \geq |L_1 u(x_1)|$ for $|x_1|<\delta$. To establish this, we first estimate $L_1 u_1$. By symmetry, we have
\begin{equation}\label{e:L1}
L_1 u_1(x_1) = c_0\int_\R \frac{u_1(x_1+y_1) - u_1(x_1)}{|y_1|^{1+\alpha}}\dd y_1,
\end{equation}
where $c_0$ depends on $\alpha,\lambda,\Lambda$, and $n$. For $x_1>1$, we have
\begin{equation}\label{e:L1est}
|L_1 u_1(x_1)| \leq c_0\int_{-\infty}^{1-x_1} \frac 2{|y_1|^{1+\alpha}}\dd y_1 = \frac{2c_0}{\alpha}\,|x_1-1|^{-\alpha}\leq C|x_1|^{-\alpha},
\end{equation}
for some constant $C$, and similarly, $|L_1 u_1(x_1)|\leq C|x_1|^{-\alpha}$ for $x_1<-1$. For $|x_1|\leq 1$, we have
\[|L_1 u_1(x_1)|\leq \int_{-\infty}^{-1-x_1} \frac 2{|y_1|^{1+\alpha}}\dd y_1 + \int_{-1-x_1}^{1-x_1} \frac{u_1(x_1+y_1)-u(x_1)}{|y_1|^{1+\alpha}}\dd y_1 + \int_{1-x_1}^\infty \frac 2{|y_1|^{1+\alpha}}\dd y_1.\]
Similarly to (\ref{e:L1est}), the first and third terms are bounded by $C|x_1|^{-\alpha}$, and since $u_1$ is smooth, the middle term is bounded by some constant, uniformly in $|x_1|\leq 1$. We conclude that $|L_1u_1(x_1)|\leq C|x_1|^{-\alpha}$ holds uniformly in $x_1\in\R$, for some $C$. Combined with the scaling (\ref{e:scaling}), this implies that for $u_r$,
\[|L_1 u_r(x_1)| = r^{-\alpha}|L_1u_1(x/r)|\leq C|x_1|^{-\alpha},\]
for some constant $C$ independent of $r$.

Next, we estimate $L_1 v$. Letting $\tilde v(x_1) = \sgn(x_1)|x_1|^{1-\alpha-\eps}$, the homogeneity $\tilde v(\lambda x_1) = \lambda^{1-\alpha-\eps}\tilde v(x_1)$ and the scaling (\ref{e:scaling}) imply that $|L_1 \tilde v(x_1)| =  C|x_1|^{1-2\alpha-\eps}$ for some $C$. For $|x_1|\leq 1$, we have
\begin{align*}
|L_1 v(x_1)| &\leq |L_1\tilde v(x_1)| + |L_1 (\tilde v - v)(x_1)|\\
& \leq C|x_1|^{1-2\alpha - \eps} + \int_{\R\setminus [-2-x_1,2-x_1]} \frac{|x_1+y_1|^{1-\alpha-\eps} - 2^{1-\alpha-\eps}}{|y_1|^{1+\alpha}}\dd y_1\\
&\leq C|x_1|^{1-2\alpha - \eps} + \int_{-\infty}^{-2-x_1} \frac {C|y_1|^{1-\alpha - \eps}}{|y_1|^{1+\alpha}}\dd y_1 + \int_{2-x_1}^{\infty} \frac {C|y_1|^{1-\alpha - \eps}}{|y_1|^{1+\alpha}}\dd y_1\\
&\leq C|x_1|^{1-2\alpha-\eps} + C\left(|x_1-2|^{1-2\alpha-\eps} + |x_1+2|^{1-2\alpha - \eps}\right)\\
 &\leq C|x_1|^{1-2\alpha-\eps}, 
\end{align*}
where $C$ denotes a changing constant.

We have $L_4 u_r(x_1) = \dfrac {C_1} r \chi_{[-r,r]}(x_1)$ and $L_4 v(x_1) = (1-\alpha - \eps)|x_1|^{-\alpha-\eps}\chi_{[-2,2]}(x_1)$. Since $L_4 u = L_4 u_r + L_4 v$ grows faster than $|L_1 u| = |L_1 u_r + L_1 v|$ as $x_1\rightarrow 0$, there is a $\delta>0$ such that $L_4 u(x_1)\geq |L_1 u(x_1)|$ when $|x_1|<\delta$, as required. This $\delta$ does not depend on $r$.

We now choose
\begin{equation}\label{e:a}
a(x) = \begin{cases}\dfrac{ L_1 u(x)}{L_4 u(x)}, & |x_1|<\delta,\\
 0, &|x_1|\geq \delta,\end{cases}
\end{equation}
so that $L u(x) = a(x)L_2 u(x)$ on $|x_1|<\delta$. On $\delta\leq |x_1|\leq 1$, we have $L u (x) = L_1u(x)$. By the above estimates, $L_1(u_r+v)$ is bounded for $\delta\leq |x_1|\leq 1$. Since $L_2 u$ is bounded for any bounded $u$, we conclude that $u(x) = u_r(x_1) + v(x_1)$ satisfies (\ref{e:bounded}). 
\end{proof}

Note that the result of the previous lemma remains true for any choice of $c(x) \in [-1,1]$, since $L_3 u(x)$ is a bounded function, for any bounded $u$. A nonzero choice of $c(x)$ will be used to make the right hand side of the equation zero and obtain our main result.

In the next lemma, we find a function $w$ such that $|L_1 w|$ is small, $L_3 w$ is large, and $L_2 w$ and $L_4 w$ cancel each other. If we add $w$ to the function $u_r+v$ from Lemma \ref{l:bounded}, these properties will allow us to choose $a(x)$ and $c(x)$ in (\ref{e:L}) such that $L(u_r+v+w) = 0$.

\begin{lemma}\label{l:w}
For any constant $C_0>0$, there exists a bounded function $w$ depending only on $x_1$, and monotonically increasing in $x_1$, satisfying
\begin{align}
L_3 w(x) - |L_1 w(x)| &\geq C_0,\label{e:w}\\
L_2 w(x) - L_4 w(x) &= 0,\label{e:w2}
\end{align}
for all $x\in B_1$, where $L_i$ are defined in \eqref{e:L123}.
\end{lemma}
\begin{proof}
Let $w_1(x_1)$ be defined by
\[w_1(x_1)  = \begin{cases} -1, &x_1<-1,\\
							x_1, &|x_1|\leq 1,\\
							1, &x_1>1,\end{cases},\]
and let $w_K(x_1) = Kw_1(x_1/K)$, where $K>2$ is a large number to be determined later. Since $w_K$ is linear when $|x_1|\leq 1$, the identity $w_K(x+y) - w_K(x) - y\cdot \nabla w_K(x) = 0$ holds there, so we have $L_2 w_K(x) - L_4 w_K(x) = 0$ in $B_1$. 

Next, we claim that for $K$ large enough, $|L_1 w_K (x_1)|$ will be uniformly bounded by an arbitrarily small constant $c$ for $|x_1|\leq 1$. By a direct computation, if $|x_1|\leq 1$, we have
\begin{align*}
L_1 w_1(x_1) &= c_0\left(\int_{-\infty}^{-1-x_1} \frac{-1-x_1}{|y_1|^{1+\alpha}}\dd y_1 + \int_{-1-x_1}^{1-x_1} \frac{y_1}{|y_1|^{1+\alpha}}\dd y_1 + \int_{1-x_1}^\infty \frac{1-x_1}{|y_1|^{1+\alpha}}\dd y_1\right)\\
&= \frac{c_0}{\alpha(1-\alpha)}  \left(|x_1-1|^{1-\alpha} - |x_1+1|^{1-\alpha}\right),
\end{align*}
with $c_0$ as in (\ref{e:L1}). Since $g(x_1) = L_1 w_1(x_1)$ is differentiable at $x_1 = 0$, we have $|g(x_1)| / |x_1| \rightarrow C$ as $x_1\rightarrow 0$, for some constant $C$. By the scaling (\ref{e:scaling}), this implies
\begin{equation}\label{e:Deltaw}
|L_1 w_K(x_1)| =   K^{1-\alpha} |L_1 w_1(x_1/K)| \leq K^{1-\alpha}C|x_1/K| = CK^{-\alpha},
\end{equation}
for $|x_1|\leq 1$, with $C$ independent of $K$. Therefore, for $K$ large enough, $|L_1 w_K(x_1)| \leq c<1$ for all $|x_1|\leq 1$ and a small constant $c$.

For $L_3 w_K$, since the integrand $(w_K(x_1+y_1) - w_K(x_1))\sgn(y_1)/|y|^{n+\alpha}$ is positive everywhere, we can write
\begin{align*}
L_3 w_K(x_1) &= \int_{\R^n\setminus B_1}\frac {w_K(x_1+y_1) - w_K(x_1)}{|y|^{n+\alpha}}\sgn(y_1) \dd y_1\\
&\geq \int_{B_{K-1}\setminus B_1} \frac {|y_1|}{|y|^{n+\alpha}} \dd y_1\\
&= \int_1^{K-1} \int_{\mathbb S^{n-1}} \frac {\rho |\theta_1|}{\rho^{n+\alpha}}  \rho^{n-1}\dd \theta \dd \rho\\
&= \frac{(K-1)^{1-\alpha} - 1}{1-\alpha} \int_{\mathbb S^{n-1}} |\theta_1| \dd \theta\\
&= C((K-1)^{1-\alpha} - 1).
\end{align*}
This lower bound holds uniformly for $|x_1|\leq 1$, so for $K$ large enough, $\inf_{|x_1|\leq 1} w_K(x_1) > 1$. This implies we can choose $K$, depending on $\alpha,n,\lambda$, and $\Lambda$, such that 
\[L_3 w_K(x_1) - |L_1 w_K(x_1)| \geq 1 - c > 0, \quad |x_1|\leq 1.\]
Therefore, given $C_0>0$, we can choose a constant $C$ such that 
\[w(x_1) := Cw_K(x_1)\]
satisfies (\ref{e:w}) and (\ref{e:w2}).
\end{proof}

We are now in a position to prove our result.

\begin{proof}[Proof of Theorem \ref{t:main}] 
Fix a modulus of continuity $\eta$. Let $\bar u = \bar u(x_1)$ be the function from Lemma \ref{l:bounded} with the corresponding kernel $K= K_1 + a(x) K_2$.

For all $x \in B_1$, we have
\begin{equation} \label{e:baru}
 -C_0 \leq L_1 \bar u(x) + a(x) (L_2 \bar u(x) - L_4 \bar u(x)) \leq C_0.
\end{equation}

Lemma \ref{l:w} implies the existence of $w(x_1)$ such that 
\begin{equation}\label{e:Lw}
\begin{aligned}
L_1 w + a(x) (L_2 w - L_4 w) + L_3 w \geq C_0,\quad |x_1|\leq 1, \\
L_1 w + a(x) (L_2 w - L_4 w) - L_3 w \leq -C_0,\quad |x_1|\leq 1.
\end{aligned}
\end{equation}
We define $u(x) = \bar u(x_1) + w(x_1)$. By \eqref{e:baru} and \eqref{e:Lw}, we have
\begin{align*}
L_1  u + a(x) (L_2 u - L_4 u) + L_3 u \geq 0, \\
L_1  u + a(x) (L_2 u - L_4 u) - L_3 u \leq 0.
\end{align*}
if $|x_1|\leq 1$. By the intermediate value theorem, there is a $c(x) \in [-1,1]$ so that
\[ L_1  u + a(x) (L_2 u - L_4 u) - c(x) L_3 u \leq 0,\]
which implies $L u(x) = 0$ in $B_1$ by (\ref{e:operator}). 

Since $w$ is also monotonically increasing in $x_1$, the oscillation $\omega_{B_r} u \geq \omega_{B_r} u_r = 1$, and $u$ breaks the modulus of continuity $\eta$.
\end{proof}

\bibliographystyle{plain}
\bibliography{counterexample}

\begin{thebibliography}{10}

\bibitem{bass2005holder}
Richard~F Bass and Moritz Kassmann.
\newblock H{\"o}lder continuity of harmonic functions with respect to operators
  of variable order.
\newblock {\em Communications in Partial Differential Equations},
  30(8):1249--1259, 2005.

\bibitem{bass2002harnack}
Richard~F Bass and David~A Levin.
\newblock Harnack inequalities for jump processes.
\newblock {\em Potential Analysis}, 17(4):375--388, 2002.

\bibitem{caffarelli2009regularity}
Luis Caffarelli and Luis Silvestre.
\newblock Regularity theory for fully nonlinear integro-differential equations.
\newblock {\em Communications on Pure and Applied Mathematics}, 62(5):597--638,
  2009.

\bibitem{chang2012regularity}
H{\'e}ctor Chang~Lara and Gonzalo D{\'a}vila.
\newblock Regularity for solutions of nonlocal, nonsymmetric equations.
\newblock In {\em Annales de l'Institut Henri Poincare (C) Non Linear
  Analysis}, volume~29, pages 833--859. Elsevier, 2012.

\bibitem{kim2012regularity}
Yong-Cheol Kim and Ki-Ahm Lee.
\newblock Regularity results for fully nonlinear integro-differential operators
  with nonsymmetric positive kernels.
\newblock {\em Manuscripta mathematica}, 139(3-4):291--319, 2012.

\bibitem{lara2012regularity}
Hector A~Chang Lara.
\newblock Regularity for fully non linear equations with non local drift.
\newblock {\em arXiv preprint arXiv:1210.4242}, 2012.

\bibitem{lara2014regularity}
H{\'e}ctor~Chang Lara and Gonzalo D{\'a}vila.
\newblock Regularity for solutions of non local parabolic equations.
\newblock {\em Calculus of Variations and Partial Differential Equations},
  49(1-2):139--172, 2014.

\bibitem{schwab2014regularity}
Russell~W Schwab and Luis Silvestre.
\newblock Regularity for parabolic integro-differential equations with very
  irregular kernels.
\newblock {\em arXiv preprint arXiv:1412.3790}, 2014.

\bibitem{silvestre2006holder}
Luis Silvestre.
\newblock Holder estimates for solutions of integro-differential equations like
  the fractional laplace.
\newblock {\em Indiana University mathematics journal}, 55(3):1155--1174, 2006.

\bibitem{silvestre2011differentiability}
Luis Silvestre.
\newblock On the differentiability of the solution to the hamilton--jacobi
  equation with critical fractional diffusion.
\newblock {\em Advances in mathematics}, 226(2):2020--2039, 2011.

\bibitem{silvestre2013loss}
Luis Silvestre, Vlad Vicol, and Andrej Zlato{\v{s}}.
\newblock On the loss of continuity for super-critical drift-diffusion
  equations.
\newblock {\em Archive for Rational Mechanics and Analysis}, 207(3):845--877,
  2013.

\bibitem{song2004harnack}
Renming Song and Zoran Vondracek.
\newblock Harnack inequality for some classes of markov processes.
\newblock {\em Mathematische Zeitschrift}, 246(1):177--202, 2004.

\end{thebibliography}
\end{document}